\documentclass{amsart}
\usepackage{amssymb}
\usepackage{amsmath}
\usepackage{amsfonts}
\usepackage{graphicx}
\usepackage{hyperref}
\usepackage{fancyhdr}
\makeatletter
\providecommand{\U}[1]{\protect\rule{.1in}{.1in}}
\newtheorem{theorem}{Theorem}

\newtheorem{corollary}[theorem]{Corollary}

\newtheorem{lemma}[theorem]{Lemma}

\newtheorem{proposition}[theorem]{Proposition}
\newtheorem{remark}[theorem]{Remark}

\begin{document}
\date{\today}
\title{Huygens Synchronization of Three Aligned Clocks}
\author{Jorge Buescu$^1$, Emma D'Aniello$^2$ and Henrique M. Oliveira$^{3*}$}
\thanks{$^1$jsbuescu@ciencias.ulisboa.pt; Mathematics Department, Faculty of Sciences, CMAFCIO - Center for Mathematics, Fundamental Applications and Operational Research, University of Lisbon, Campo Grande, 1749-006 Lisbon, Portugal.}
\thanks{$^2$emma.daniello@unicampania.it; ORCID: 0000-0001-5872-0869 Dipartimento di Matematica e Fisica, Universit\`{a} degli Studi della Campania
\textquotedblleft Luigi Vanvitelli\textquotedblright, Viale Lincoln n. 5 -
81100 Caserta, Italia}
\thanks{$^3$henrique.m.oliveira@tecnico.ulisboa.pt; ORCID: 0000-0002-3346-4915
Department of Mathematics \and Center for Mathematical Analysis, Geometry and
Dynamical Systems, Instituto Superior T\'ecnico, University of Lisbon, Av.
Rovisco Pais, 1049-001, Lisbon, Portugal;\\$^*$ Corresponding author;  }

\subjclass{Primary 34D06, Secondary 37E30}
\keywords{Synchronization of oscillators, Stability, Andronov pendulum clocks, Mutual symmetric impact interaction}
\thanks{The author JB was partially supported by  Fundação para a Ciência e
Tecnologia, Fundação para a Ciência e Tecnologia, UIDB/04561/2020. The author ED was partially supported by the program Erasmus+, the project 2024 DYNAMIChE of the INdAM Group GNAMPA, the project PRIN 2022 QNT4GREEN, and the group UMI-TAA: Approximation Theory and Applications. The author HMO was partially supported by  Fundação para a Ciência e
Tecnologia, UIDB/04459/2020 and UIDP/04459/2020.
\\
Disclosure of interest: The authors report no conflict of interest.
}

\begin{abstract}
This study examines the synchronization of three identical oscillators
arranged in an array and coupled by small impacts, wherein each oscillator
interacts solely with its nearest neighbor. The synchronized state, which is
asymptotically stable, is
characterized by phase opposition among alternating oscillators. We analyze the system 
using a non-linear discrete dynamical system based on a difference equation derived from 
the iteration
of a plane diffeomorphism. We illustrate these results with the application to a system of three aligned Andronov clocks, showcasing their applicability to a broad range of oscillator systems.

\end{abstract}

\maketitle




\section{Introduction}

Synchronization of oscillators is an extremely rich topic in Dynamical Systems, being treated by a very substantial body of literature across diverse 
research domains \cite{Luo2009, Luo2013, Pit, strogatz2004}.  The recently published book \cite{GolStewrt2023} 
provides a general framework for synchronization in symmetrical networks. 
In a recent publication \cite{EH2}, two of the authors investigated the
synchronization of three oscillators interacting symmetrically. The model
incorporated the Andronov \cite{And} pendulum clock, as used in \cite{OlMe}.
The theory presented in \cite{EH2} relies only on systems having limit cycles
and small interactions between oscillators once per cycle, ensuring
applicability irrespective of specific details of the oscillator models. This paper focused on
constructing a model and its subsequent analysis, yielding two possible locked
states dependent on initial conditions. The approach involved deriving a
planar difference equation for independent phase differences $(x,y)$ relative
to a reference oscillator (that the authors called clock $A$), with the resulting map
being a diffeomorphism of the plane:%

\begin{equation}%
\begin{array}
[c]{cccc}%
\Omega: &
\mathbb{R}^{2} & \ \ \ \longrightarrow  \ \  \ 
\mathbb{R}
^{2}  &  \\
&
\begin{bmatrix}
x\\
y
\end{bmatrix}
& \longmapsto &
\begin{bmatrix}
x+2a\sin x+a\sin y+a\sin(x-y)\\
y+a\sin x+2a\sin y+a\sin(y-x)
\end{bmatrix}
.
\end{array}
\label{eq:01}%
\end{equation}

In the context of this proposed \textit{triangle model}, the dynamics corresponding to 
 \eqref{eq:01} exhibits two asymptotic attractors with a $120^{\circ}$
phase difference, occurring either clockwise or counterclockwise, depending on
initial conditions.

Motivated by the mentioned previous work, the present study explores the behavior of a
linear arrangement of oscillators, investigating interactions solely between
nearest neighbors. Instead of a triangle model, we now deal with a string of
three oscillators and address questions regarding the synchronization
dynamics. In this paper, we do not focus on the construction of the model,
already performed in \cite{EH2} in the general setting; instead, we explore directly 
the dynamical system corresponding to the linear arrangement of oscillators
with nearest neighbor interaction.

The configuration involves three oscillators (\textit{B-A-C}) along a line
segment. Oscillator $B$ interacts only with the central oscillator $A$;
 the central oscillator $A$ interacts with $B$ and $C$; and oscillator $C$
interacts only with $A$. Thus $B$ and $C$ remain unconnected by a direct
perturbation. This nearest neighbor setting, common in various scientific
disciplines, provides a reasonable approximation of oscillators placed in a
linear arrangement, particularly for mechanical or electronic systems where
interactions decay rapidly with distance.

Unlike the symmetrical triangle model, the linear model presents additional
complexity in analysis due to the absence of certain symmetries and trivial
invariant sets in the phase space as explored in \cite{EH2}. While the
triangle model benefited from straight-line saddle-node heteroclinics, the
linear model lacks such features.

It is essential to note that the model we present is {\em perturbative}, excluding
macroscopic momentum exchange between oscillators. This contrasts with
scenarios involving movable supports, as explored in
\cite{Benn,Col,Frad,Jova,Col2,Martens,Oud,Sen}, which involve permanent
momentum exchange. Moreover, in the present article, we consider that the
oscillators are modeled by second-order differential equations with limit
cycles, which is different from the \textit{integrate-and-fire synchronization} in
\cite{Campbell1997,Campbell1999,Kuramoto1975,Mirollo1990,Strogatz2000}, where
first-order equations are employed. This perturbative, or long-time
synchronization, i.e., secular, approach aligns with previous works such as
\cite{Abr2,Abr,EH2,OlMe,Vass}.

This paper is organized as follows. In the second section, we
recall from \cite{EH2} a model for the phase difference in a row of three
identical pendulum clocks attached to a rigid wall, resulting in a
two-dimensional discrete dynamical system. In the third section, we study the
properties of invariance and symmetry of this dynamical system, reducing the
study to an invariant set. In the fourth section, we analyze the asymptotic
stability in the specific invariant set, obtaining the fundamental result. The
concluding section summarizes the results, suggests generalizations, and
outlines future research.

\section{Model for Three Aligned Clocks}

We consider 
a system composed of three identical oscillators in a linear arrangement, aligned so that the left one and the right one are at the same distance from the middle one, and with a weak nearest-neighbor interaction. We refer to these oscillators as clocks, since we assume isochronism of each oscillator when isolated from perturbations, and we may conceive they are mounted on the same wall. We may assume, in particular, that
those clocks obey the classical Andronov model \cite{And}; however, the
synchronization theory we develop is independent
of the particular model of ODEs governing the oscillators. Indeed, 
any system of three identical planar ODEs with a limit cycle with weak burst pulses perturbing
nearest neighbors once per turn satisfies the hypotheses of this article.

In \cite{EH2} we discuss the case of Andronov oscillators disposed in a triangular, symmetrical setting. When  clock $A$ attains a critical position in its
limit cycle, it receives an internal kick from its escape mechanism
\cite{OlMe}; the impact then propagates along the wall, slightly perturbing both
 clocks $B$ and $C$. The same happens in turn when $B$ or $C$ arrive at
the critical position in their corresponding limit cycles, interacting always with the
other two clocks.
The perturbation is assumed to be instantaneous since the time of travel of 
the mechanical waves in the wall between the clocks, corresponding to the speed of sound, is very small compared to the period of the clocks, as discussed in \cite{OlMe,EH2}.

In the construction of the dynamical system for phase differences of Andronov
clocks arranged in a line with nearest neighbor interaction, we assume the
hypotheses of \cite{EH2}, the only exception being, as mentioned above, that there is no direct
interaction between oscillators $B$ and $C$ . This is the
only, but crucial, difference between this model and the all-to-all interaction
between oscillators. In the case of interaction between all
 three clocks, it is shown in \cite{EH2} that, denoting the 
two independent phase differences between oscillators $B$ and $C$ relative to the central clock $A$ respectively by $x$ and $y$,  the dynamics is described by the discrete system
\begin{equation}%
\begin{bmatrix}
x_{n+1}\\
y_{n+1}%
\end{bmatrix}
=\Omega%
\begin{bmatrix}
x_{n}\\
y_{n}%
\end{bmatrix}
=%
\begin{bmatrix}
x_{n}+2a\sin x_{n}+a\sin y_{n}+a\sin\left(  x_{n}-y_{n}\right) \\
y_{n}+a\sin x_{n}+2a\sin y_{n}+a\sin\left(  y_{n}-x_{n}\right)
\end{bmatrix}
. \label{eq: 02}%
\end{equation}

\noindent Notice that the terms $\left(  x_{n}-y_{n}\right)  $ in equations
(\ref{eq: 02}) account for the direct interaction between the outer clocks $B$
and $C$. Since in the linear nearest neighbor
setting this interaction does not exist,  the corresponding model is obtained 
 simply by setting the corresponding term to zero. We thus obtain the equations for three aligned clocks with nearest neighbor interaction:%

\begin{equation}%
\begin{bmatrix}
x_{n+1}\\
y_{n+1}%
\end{bmatrix}
=F%
\begin{bmatrix}
x_{n}\\
y_{n}%
\end{bmatrix}
=%
\begin{bmatrix}
x_{n}+2a\sin x_{n}+a\sin y_{n}\\
y_{n}+a\sin x_{n}+2a\sin y_{n}%
\end{bmatrix}
. \label{Eq3}%
\end{equation}

\noindent It will be useful, for later purposes, to write this system as a perturbation of the identity. Defining
\begin{equation}
\label{eq_f}
f(x,y)= x + 2a \sin x + a \sin y
\end{equation}
and
\begin{equation}
\label{eq_varphi}
\varphi(x,y)=  2 \sin x + \sin y
\end{equation}
we may write system \eqref{Eq3} as

\begin{align*}
F%
\begin{bmatrix}
x_n\\
y_n
\end{bmatrix}
&  =%
\begin{bmatrix}
f\left(  x_{n},y_{n}\right) \\
f\left(  y_{n},x_{n}\right)
\end{bmatrix}
\\
&  =%
\begin{bmatrix}
1 & 0\\
0 & 1
\end{bmatrix}%
\begin{bmatrix}
x_{n}\\
y_{n}%
\end{bmatrix}
+a%
\begin{bmatrix}
\varphi(x_{n},y_{n})\\
\varphi(y_{n},x_{n})
\end{bmatrix}
.
\end{align*}

We 
consider the parameter $0 < a \ll 1$ small\footnote{For the case of the Andronov clock, it is possible to show that 
$a=\frac{\alpha\mu}{8h^{2}}$, where $\alpha$ is the coupling constant, with
units of force, $\mu$ is the dry friction coefficient of the gears of the
internal mechanism of the clocks, and $\frac{h^{2}}{2}$ is the kinetic energy
given to the pendulum in a complete turn, which compensates the energy loss
due to friction.}, as discussed in detail in \cite{EH2,OlMe}. 

\subsection{Basic Analysis of the Model}

We proceed by analyzing the properties of the dynamical system of the phase
differences of the three aligned clocks.

\paragraph{Jacobian Matrix}

The Jacobian matrix of $F$ is given by
\[%
DF(x,y) = \, 
\begin{bmatrix}
1+2a\cos x & a\cos y\\
a\cos x & 1+2a\cos y
\end{bmatrix}
,
\]
with Jacobian determinant
\[
JF\left(  x,y\right)  =1+2a\left(  \cos y+\cos x\right)  +3a^{2}\cos x\cos y.
\]
Thus, the Jacobian is greater than $0$ for small positive $a$. Indeed, 
 it may be easily shown that this holds for all $x,y$ if $a < \frac{1}{6}$. The inverse function theorem then ensures that, for these values of the parameter  $a$, the map $F$ is an analytic diffeomorphism. From this point onwards we shall suppose throughout that $0 < a < \frac{1}{6}$.

The eigenvalues of $JF(x,y)$ are given by the characteristic equation
\[
(1+2a\cos x-\lambda)(1+2a\cos y-\lambda)-{a}^{2}\cos x\cos y=0\text{,}%
\]
or equivalently
\[
\lambda^{2}-2\lambda\left[  1+a\left(  \cos x+\cos y\right)  \right]
+2a\left(  \cos y+\cos x\right)  +3a^{2}\cos x\cos y+1=0.
\]

\noindent The map $F$ has four families of fixed points indexed by two integers $k, l \in \mathbb{Z}$:
\[
\begin{array}{lll}
P_0^{l,k} & = & \left(  2l\pi ,2k\pi \right) \\
P_1^{l,k} & = & \left(  2l\pi ,\pi + 2k\pi \right) \\
P_2^{l,k} & = & \left( \pi + 2l\pi ,  2k\pi \right) \\
P_3^{l,k} & = & \left( \pi + 2l\pi , \pi + 2k\pi \right)
\end{array} \]

%
%
%

\noindent A  the fixed points and the Jacobian matrix are $2\pi$-periodic in both coordinates, we select for each family of fixed points  the representative with $k = l =0$ (dropping the upper indices) and perform the  linear stability analysis, which then carries over to the corresponding family indexed by $k,l \in \mathbb{Z}$.

 For $P_0$ the eigenvalues are ${\lambda}_{1}=1+a$ and ${\lambda}_{2}=1+3a$, 
and thus $P_0$ is a hyperbolic repellor. At both $P_1$ and $P_2$ the eigenvalues are 
 ${\lambda}_{1}=1+a\sqrt{3}$ and ${\lambda}_{2}=1-a\sqrt{3}$; hence, both are hyperbolic saddle points. At $P_3$ the eigenvalues are ${\lambda
}_{1}=1-3a$ and ${\lambda}_{2}=1-a$; hence, it is a local hyperbolic attractor. 

In the table below we summarize the local classification of the fixed points of each family, for $k,l \in \mathbb{Z}$.

\begin{center}%
\begin{tabular}
[c]{|c|c|c|c|}\hline
Fixed Points & $P_0^{l,k} $ & $P_1^{l,k} $ and $P_2^{l,k} $ & $P_3^{l,k} $\\\hline
Eigenvalues & $1+a,1+3a$ & $1+a\sqrt{3},1-a\sqrt{3}$ & $1-3a,1-a$\\\hline
Classification & Source & Saddle & Sink\\\hline
\end{tabular}

\end{center}

We observe that at the saddle $(\pi,0)$, the eigenvectors of the Jacobian matrix of $F$ define
stable $v_{s}$ and unstable $v_{u}$ directions  by%
\begin{align}
v_{s}\left(  \pi,0\right)   &  =\left(  2+\sqrt{3},1\right)
,\label{eq: stableEV}\\
v_{u}\left(  \pi,0\right)   &  =(2-\sqrt{3},1). \label{eq: unstableEV}%
\end{align}

Since 
$F$ is an invertible map, orbits corresponding to different initial conditions in 
$\mathbb{R}^{2}$ either coincide or do not intersect. There are invariant sets under the dynamics of $F$ which are useful in
the proof of the existence of the attractor $P_3=(\pi,\pi)$; in the sequel, we
construct those sets using the symmetries of $F$.

In what follows we adopt the standard topological notations: given a set $B\subset\mathbb{R}^{2}$,  the interior of $B$ is denoted by $\text{int}(B)$, the closure of $B$ is denoted by $\overline{B}$, and the boundary  of $B$ is denoted by $\partial(B)$.

The main result of this paper is the existence of a bounded, simply
connected open set $A$ in the plane containing $P_3 =(\pi,\pi)$ 
which is the basin of attraction of $P_3$ with the following properties:
\begin{enumerate}
\item the points $(0,0)$,
$(0,2\pi)$, $(2\pi,0)$, and $(2\pi,2\pi)$ belong to the boundary of $A$ and
are repellers in the closure $\overline{A}$, and the points $(0,\pi)$,
$(\pi,0)$, $(\pi,2\pi)$, and $(2\pi,\pi)$ belong to the boundary of $A$ and
are saddles in $\overline{A}$;
\item the boundary of the set $A$ is strongly invariant under
the dynamics of $F$, that is, $F(\partial A)=\partial A$;
\item the collection of plane translations  $\overline{A}^{l,k} = \overline{A} + (2\pi l, 2\pi k)$ of $\overline{A}$ is a tiling of the plane $\mathbb{R}^2$.
\end{enumerate}
The double periodicity of the map $F$ then ensures that in each set $\overline{A}^{l,k}$ 
the dynamics is the same as the dynamics in $\overline{A}$.

\paragraph{Nullclines}

Recalling 
 the notation in \eqref{eq_varphi}, consider  the two implicit curves
which are easily seen to exist in the parameter range of interest) defined by
$
\varphi\left(  x,y\right)  =0
$
and
$
\varphi\left(  y,x\right)  =0.
$
When $\varphi\left(  x,y\right)  =0$, the perturbation field $F$ in \eqref{Eq3} is vertical; when $\varphi\left(  y,x\right)  =0$, the perturbation field $F$ is horizontal. 
It follows that these implicit curves represent the nullclines of $F$.
In Figure \ref{fig: 01}, we plot these nullclines, indicating also the sign of the
perturbation field. Although the sets bounded by the nullclines are not invariant, 
the sign variations strongly suggest that $P_3= (\pi,\pi)$ has the properties
enumerated above.

\begin{figure}[tbh]
\centering
\includegraphics[height=3.556in,width=3.513in]{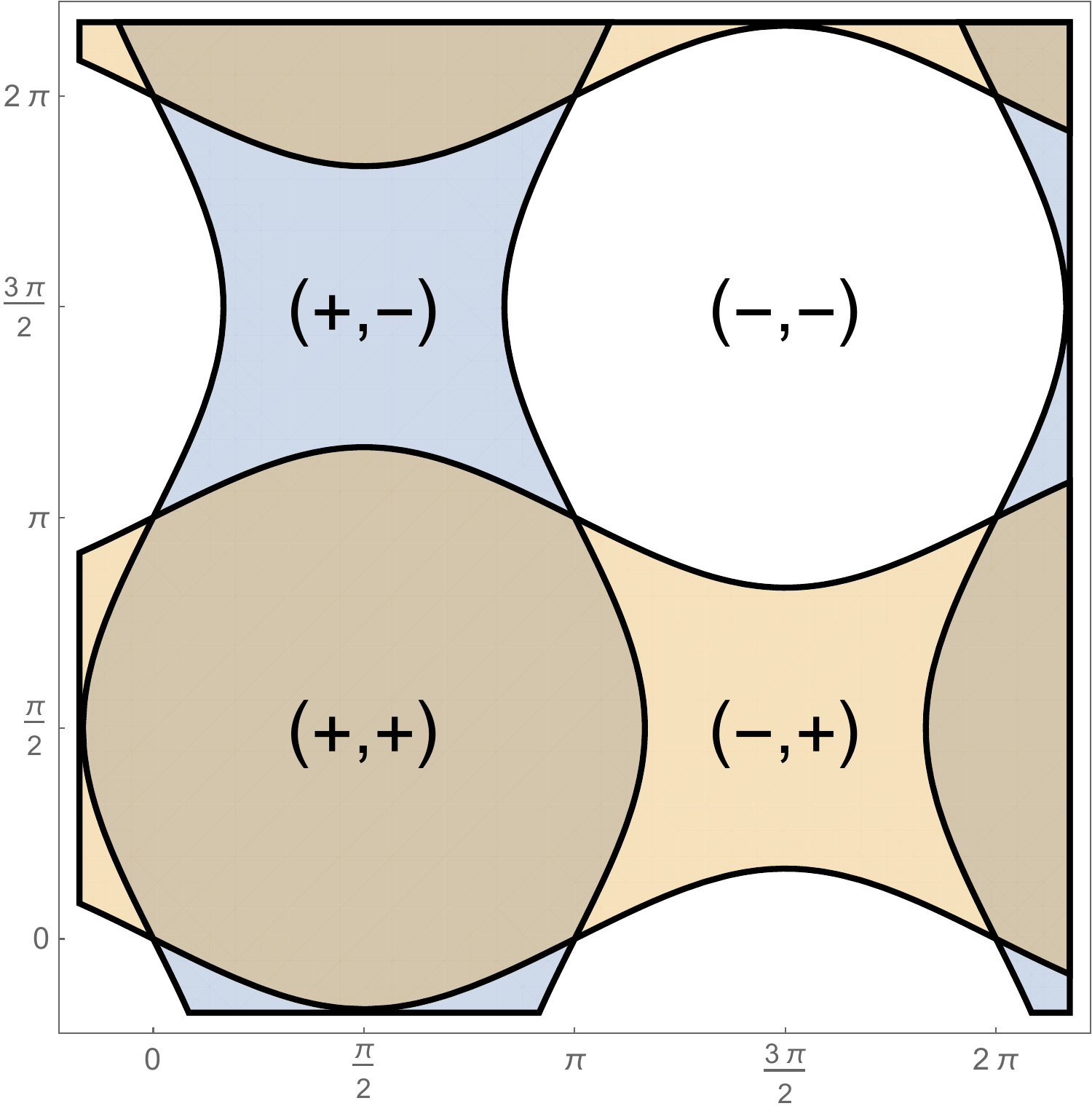}\caption{Nullclines
of the field $F$ for $a=0.1$.}%
\label{fig: 01}%
\end{figure}

\newpage

\section{Invariance and symmetry}
\label{sec_invariance_symmetry}

\subsubsection{Invariant lines}

In order to show the desired result, we first determine two one-parameter families of invariant straight lines as shown in Figure \ref{fig: 02}.

\begin{figure}[ptb]
\begin{center}
\includegraphics[
height=4.4503in,
width=4.4008in
]{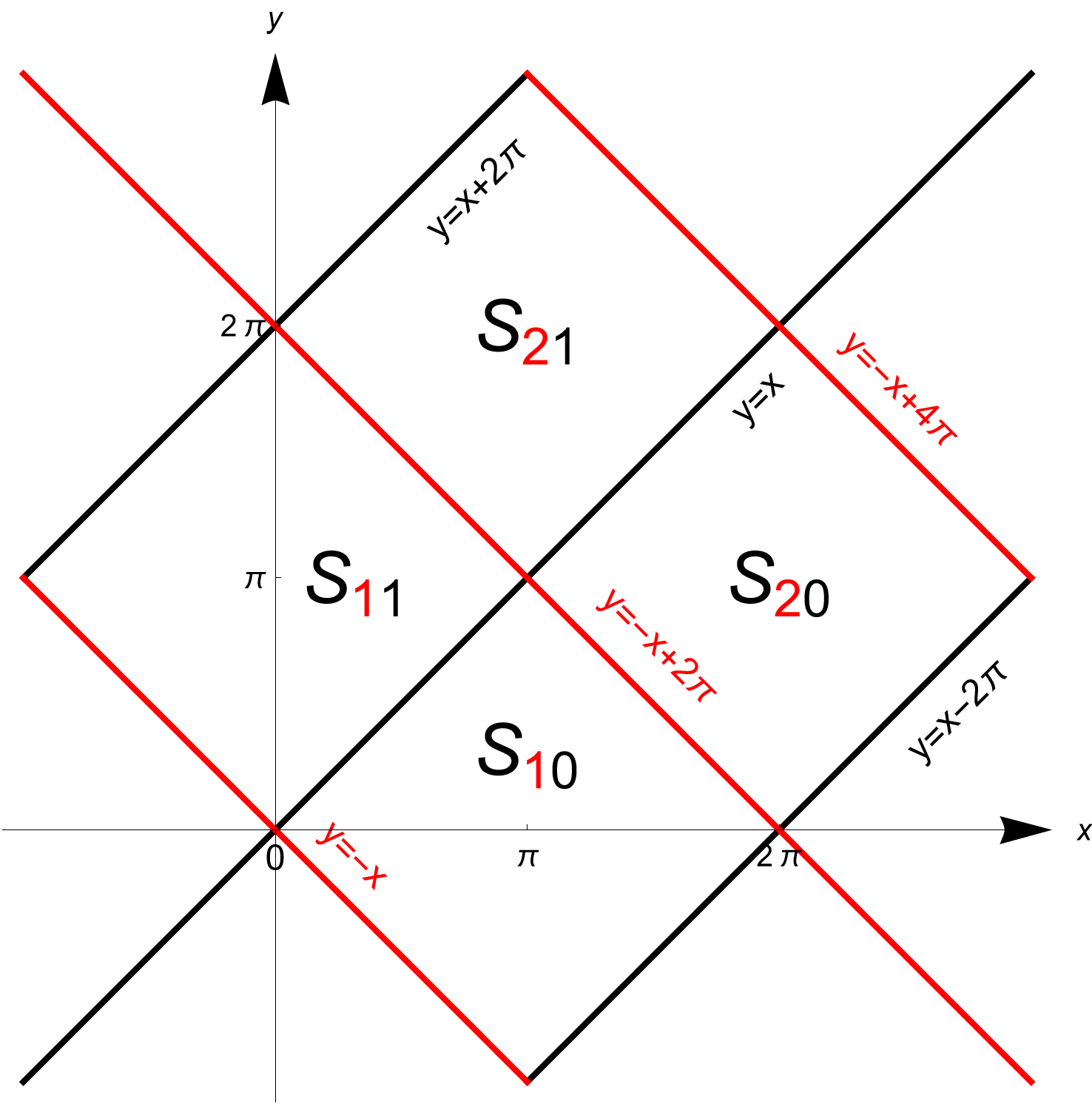}
\end{center}
\caption{Invariant lines near $\left(  \pi,\pi\right)  .$}%
\label{fig: 02}%
\end{figure}

\begin{lemma}
Consider the one-parameter families of straight lines $\Gamma_{k}^{+} = \{ (x,y): 
y=x+2k\pi \}_{k\in\mathbb{Z}}$  and  $\Gamma_{l}^{-} = \{ (x,y): 
y=-x+2l\pi \}_{l\in\mathbb{Z}}$. Then:
\begin{enumerate}
\item[(1)] For each $k \in \mathbb{Z}$ the straight line $\Gamma_{k}^{+}$ is strongly invariant, that is
$ F(\Gamma_{k}^{+}) = \Gamma_{k}^{+}. $
\item[(2)] For each $l \in \mathbb{Z}$ the straight line $\Gamma_{l}^{-}$ is strongly invariant, that is
$ F(\Gamma_{l}^{-}) = \Gamma_{l}^{-}. $
\end{enumerate}
\label{lemma_1}
\end{lemma}

\begin{proof} Recalling the definition of $f$ from \eqref{eq_f}, the proof of (1) follows from a simple computation: given $(x, x+2k\pi) \in \Gamma_{k}^{+}$, 
\begin{align*}
F\left(  \begin{bmatrix}x\\x+2k\pi\end{bmatrix}\right)   &=\begin{bmatrix}x+3a\sin x\\x+3a\sin x+2k\pi\end{bmatrix} \\
&=\begin{bmatrix}f\left(  x,x\right) \\f\left(  x,x\right)  +2k\pi\end{bmatrix}
\end{align*} is also in $\in \Gamma_{k}^{+}$. The proof of (2) is similar, so details are omitted.
\end{proof}

%

\noindent The next result is now a simple consequence of the previous lemma.

\begin{proposition}
The squares $S_{lk}$ defined by
\[
S_{lk}=\left\{ (x,y) :   -x + 2(l-1)\pi \leq y \leq -x + 2l\pi , \   x + 2(k-1)\pi \leq y \leq x + 2k\pi    \right\}
\]
where  ${l,k \in \mathbb{Z}}$, are invariant sets for the dynamics of $F$.
\end{proposition}

Due to the invertible nature of $F$, different orbits of the dynamical system do not intersect.  Fix now the integers $l$ and $k$. Consider the curves $\Gamma_{l-1}^{-}$, $\Gamma_{l}^{-}$, $\Gamma_{k-1}^{+}$, and $\Gamma_{k}^{+}$; these four families of curves define the square $S_{lk}$ in $\mathbb{R}^{2}$ whose boundary cannot be crossed by the orbits with initial conditions in the interior of $S_{lk}$.

\subsubsection{Equivariance}

Consider a linear bijection  $\Phi \in GL( \mathbb{R}^2)$ which commutes with $F$, that is 
\[
 F\circ\Phi \, (x,y)=\Phi\circ F \, (x,y) \ \ \ \forall (x,y) \in \mathbb{R}^2.
\]
Then, $F$ is  said to be a $\Phi$-equivariant map\cite{auslander,GS1}. The set of all such $\Phi$ is easily seen to form a group under composition. This group is a linear action of the symmetry group $G$ of the map $F$; with a slight abuse of language we identify this representation with $G$ itself, so that
\[ G = \{ \Phi \in GL(\mathbb{R}^2) : F\circ\Phi \, =\Phi\circ F  \}. \]

The following Proposition summarizes some standard results in equivariant dynamics of which we shall make extensive use; for completeness, we state and prove it in the present context. 
Recall that a set $S$ is strongly $F$-invariant if $F(S)=S$, while an orbit with initial condition $X_0$ is the set  $O_{X_{0}}= \{ \left(
X_{n}\right) \}_{n \in \mathbb{Z}}$ such that
\[
X_{n+1}=F\left(  X_{n}\right)  ,\quad n \in \mathbb{Z}.
\]

\begin{proposition}
Let $F: \mathbb{R}^2 \to \mathbb{R}^2$ be a $\Phi$-equivariant diffeomorphism. Then:
\begin{enumerate}
\item If the set $S$ is strongly $F$-invariant, then  the set $\Phi\left(  S\right)  $ is also strongly $F$-invariant.  
\item If $O_{X_{0}}$ is an orbit of $F$ with initial condition $X_{0}$, then
\[
\Phi\left(  O_{X_{0}}\right)  = \{ (  \Phi (  X_{n}) \}_{n \in \mathbb{Z}}
\]
is  an orbit of $F$ with initial condition $\Phi\left(  X_{0}\right)  $.
\end{enumerate}
\label{prop_invariance}
\end{proposition}

\begin{proof}
If $S$ is strongly invariant, that is $F(S) = S$, 
then $\Phi$-equivariance immediately implies 
\[
F\left(  \Phi\left(  S\right)  \right)  =\Phi\left(  F\left(  S\right)
\right)  =\Phi\left(  S\right) 
\]
showing invariance of $\Phi\left(  S\right)$ and proving the first statement. 

For the second statement, notice that $\Phi$-equivariance of $F$ implies 
\[
F^{n}\circ\Phi\left(  X\right)  =\Phi\circ F^{n}\left(  X\right)   \ \ \forall n \in \mathbb{Z}
\]
and therefore, if $Y_0 = \Phi(X_0)$, then 
\[
F^n(Y_0) = F^n(\Phi(X_0)) = \Phi(F^n(X_0)) \ \ \forall n \in \mathbb{Z}, \]
finishing the proof.
\end{proof}

Consider all the orbits with initial conditions in an invariant set $S$.
Proposition \ref{prop_invariance} above implies that an orbit in $S$ has an equivalent
orbit, in the sense of linear conjugation, in the invariant set $\Phi\left(  S\right)  $. 
More generally, the dynamics of each initial condition in $S$ is linearly conjugate to 
the dynamics of the corresponding initial condition in $\Phi\left(  S\right)  $. 
In other words, the flow of the
dynamical system in $S$ is linearly conjugate to the flow in $\Phi\left(  S\right)  $.

\subsubsection{Symmetry\label{Subsubsec:symmetry}}

We now construct explicitly the elements of the linear symmetry group $G$ of the diffeomorphism $F$.

\begin{proposition}
(Translations) The translation transformations $T_{lk}$ defined by%
\[%
\begin{array}
[c]{cccc}%
T_{lk}: & \mathbb{R}^{2} & \longrightarrow & \mathbb{R}^{2}\\
& \left[
\begin{array}
[c]{c}%
x\\
y
\end{array}
\right]  & \longmapsto & \left[
\begin{array}
[c]{c}%
2l\pi+x\\
2k\pi+y
\end{array}
\right]  \text{,}%
\end{array}
\]
commute with $F$.
\label{prop_translations}
\end{proposition}

\begin{proof}
Recalling the notation for $f$ introduced in \eqref{eq_f}, the proof is a simple computation:
\begin{align*}
F\left(  T_{lk}\left(  \left[
\begin{array}
[c]{c}%
x\\
y
\end{array}
\right]  \right)  \right)   &  =F\left(  \left[
\begin{array}
[c]{c}%
2l\pi+x\\
2k\pi+y
\end{array}
\right]  \right) \\
&  =\left[
\begin{array}
[c]{c}%
2l\pi+x+2a\sin x+a\sin y\\
2k\pi+y+a\sin x+2a\sin y
\end{array}
\right] \\
&  =\left[
\begin{array}
[c]{c}%
2l\pi\\
2k\pi
\end{array}
\right]  +\left[
\begin{array}
[c]{c}%
f\left(  x,y\right) \\
f\left(  y,x\right)
\end{array}
\right] \\
&  =T_{lk}\left(  F\left(  \left[
\begin{array}
[c]{c}%
x\\
y
\end{array}
\right]  \right)  \right)  \text{.}%
\end{align*}
\end{proof}

\noindent In view of the translation symmetry of $F$, it is enough to compute the dynamics in a fundamental domain where the dynamics is invariant to obtain the dynamics in the whole of $\mathbb{R}^{2}$. 

\begin{proposition}
(Rotation by $\pi$) The transformation $\Sigma$ defined by%
\[%
\begin{array}
[c]{cccc}%
\Sigma: & \mathbb{R}^{2} & \longrightarrow & \mathbb{R}^{2}\\
& \left[
\begin{array}
[c]{c}%
x\\
y
\end{array}
\right]  & \longmapsto & \left[
\begin{array}
[c]{c}%
-x\\
-y
\end{array}
\right]  %
\end{array}
\]
commutes with $F$.
\end{proposition}
\begin{proof}
The proof is again a simple verification:%
\begin{align*}
F\left(  \Sigma\left(  \left[
\begin{array}
[c]{c}%
x\\
y
\end{array}
\right]  \right)  \right)   &  =F\left(  \left[
\begin{array}
[c]{c}%
-x\\
-y
\end{array}
\right]  \right) \\
&  =\left[
\begin{array}
[c]{c}%
-x-2a\sin x-a\sin y\\
-y-a\sin x-2a\sin y
\end{array}
\right] \\
&  =-\left[
\begin{array}
[c]{c}%
f\left(  x,y\right) \\
f\left(  y,x\right)
\end{array}
\right] \\
&  =\Sigma\left(  F\left(  \left[
\begin{array}
[c]{c}%
x\\
y
\end{array}
\right]  \right)  \right)  \text{.}%
\end{align*}
\end{proof}
Finally, there are two more linear transformations that commute with $F$.

\begin{proposition}
(Reflection over the lines $y=x$ and $y=-x$) The transformations $\Psi^{\pm
}\left(  x,y\right)  $, such that%
\[%
\begin{array}
[c]{cccc}%
\Psi^{\pm}: & \mathbb{R}^{2} & \longrightarrow & \mathbb{R}^{2}\\
& \left[
\begin{array}
[c]{c}%
x\\
y
\end{array}
\right]  & \longmapsto & \left[
\begin{array}
[c]{c}%
\pm y\\
\pm x
\end{array}
\right]  %
\end{array}
\]
commute with $F$.
\end{proposition}

\begin{proof}
Once more the proof is a simple verification:%
\begin{align*}
F\left(  \Psi^{\pm}\left(  \left[
\begin{array}
[c]{c}%
x\\
y
\end{array}
\right]  \right)  \right)   &  =F\left(  \left[
\begin{array}
[c]{c}%
\pm y\\
\pm x
\end{array}
\right]  \right) \\
&  =\left[
\begin{array}
[c]{c}%
\pm y\pm2a\sin y\pm a\sin x\\
\pm x\pm a\sin y\pm2a\sin x
\end{array}
\right] \\
&  =\left[
\begin{array}
[c]{c}%
\pm f\left(  y,x\right) \\
\pm f\left(  x,y\right)
\end{array}
\right] \\
&  =\Psi^{\pm}\left(  F\left(  \left[
\begin{array}
[c]{c}%
x\\
y
\end{array}
\right]  \right)  \right).
\end{align*}
\end{proof}
\begin{remark}
These four geometrical elements in the previous propositions are of course the generators of the symmetry group $G$ of $F$. For the purposes of this paper we will not need the general tools from group theory, so we abstain from further developing the characterization of $G$.
\end{remark}

\subsubsection{Simplification of the dynamics}

In view of the symmetry considerations, the four squares adjacent to the hyperbolic attractor $\left(  \pi,\pi\right)  $, labeled 
$S_{10}$, $S_{21}$, $S_{11}$, and $S_{20}$ as
shown in Figure \ref{fig: 04},  have equivalent dynamics. 
\begin{figure}[tbh]
\centering
\includegraphics[height=3.5568in,width=3.6403in]{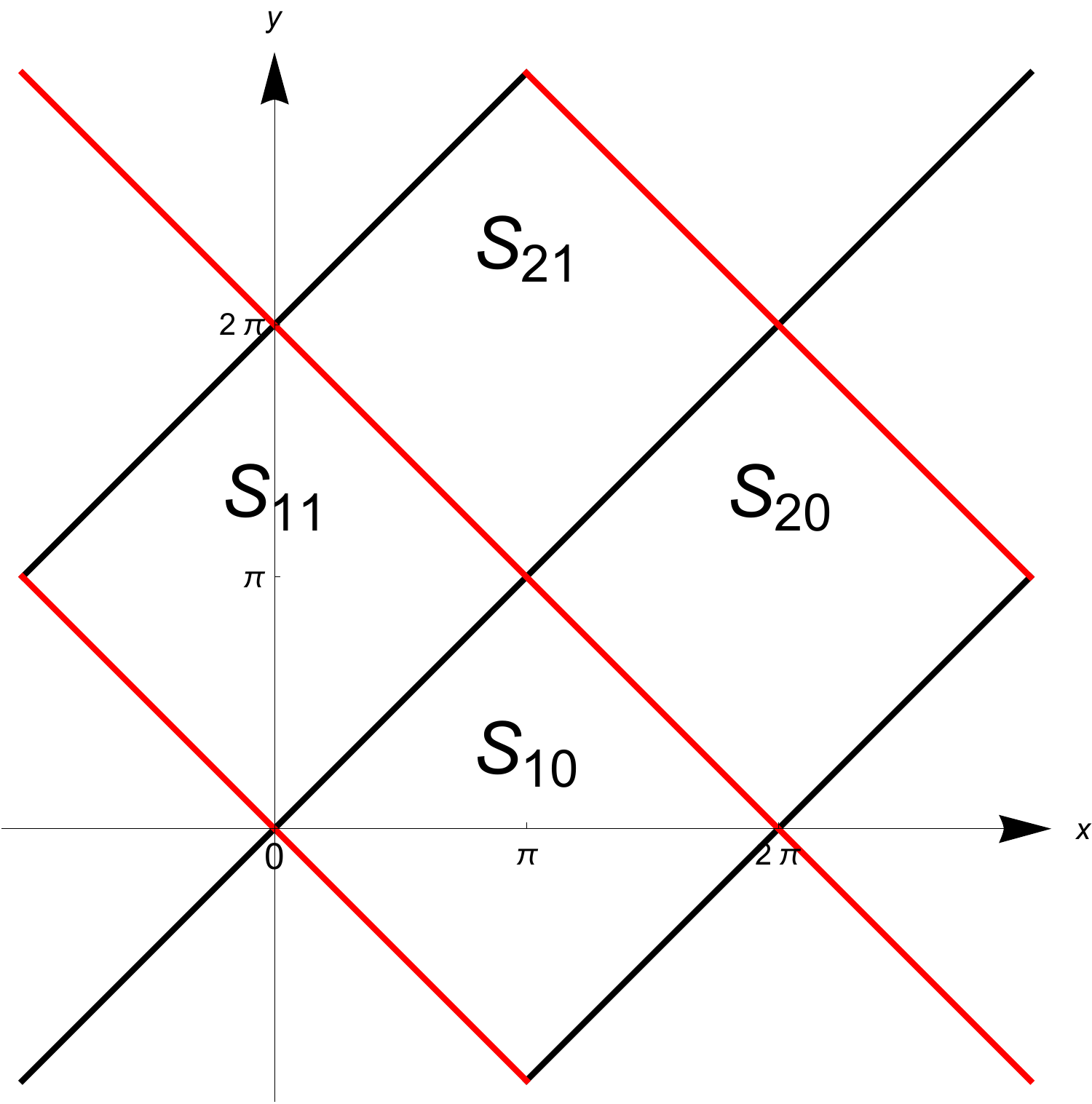}\caption{The
four squares.}%
\label{fig: 04}%
\end{figure}

We focus on the square $S_{10}$ with vertices $\left(  0,0\right)  $, $\left(
\pi,\pi\right)  $, $\left(  2\pi,0\right)  $, and $\left(  \pi,-\pi\right)  $,
defined by
\begin{equation}
S_{10}=\left\{  \left(  x,y\right)  :-x\leq y\leq2\pi-x\wedge x-2\pi\leq y\leq
x\right\} 
\label{S10}
\end{equation}
as depicted in Figure \ref{fig: 03}.

\begin{figure}[tbh]
\centering
\includegraphics[height=3.557in,width=3.6452in]{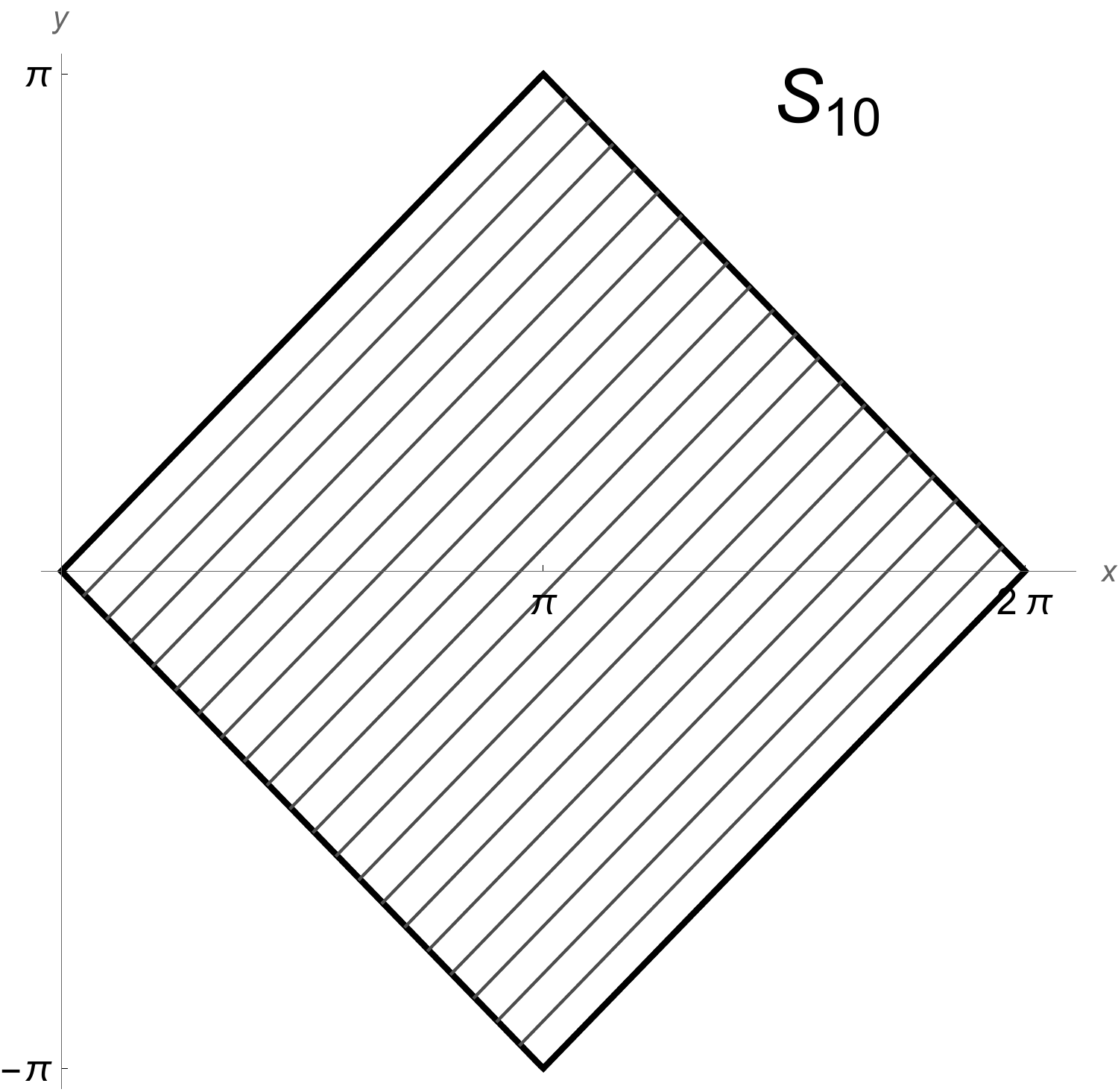}\caption{The
square $S_{10}$.}%
\label{fig: 03}%
\end{figure}


The vertically shifted square relative to $S_{10}$ is the square $S_{21}$,
defined as
\begin{equation}
S_{21}=T_{01}S_{10}, \label{eq: sym1}%
\end{equation}
where
\[
T_{01}\left[
\begin{array}
[c]{c}%
x\\
y
\end{array}
\right]  = \left[
\begin{array}
[c]{c}%
x\\
2\pi+y
\end{array}
\right]  ,
\]
with vertices $\left(  0,2\pi\right)  $, $\left(  \pi,3\pi\right)  $, $\left(
2\pi,2\pi\right)  $, and $\left(  \pi,\pi\right)  $.

The square $S_{11}$ with vertices $\left(  0,0\right)  $, $\left(  \pi
,\pi\right)  $, $\left(  0,2\pi\right)  $, and $\left(  -\pi,\pi\right)  $, is
related to $S_{10}$ by the relationship
\begin{equation}
S_{11}=\Psi^{+}S_{10}. \label{eq: sym2}%
\end{equation}

The square $S_{20}$ with vertices $\left(  2\pi,0\right)  $, $\left(  3\pi
,\pi\right)  $, $\left(  2\pi,2\pi\right)  $, and $\left(  \pi,\pi\right)  $,
is given by
\begin{equation}
S_{20}=T_{10}S_{11}=T_{10}\Psi^{+}S_{10}. \label{eq: sym3}%
\end{equation}

These relations coupled with equivariance of the map $F$ imply that the global dynamics
in $\mathbb{R}^2$ is completely determined by the dynamics in $S_{10}$. It therefore suffices to analyze the dynamics in $S_{10}$ to understand the global synchronization dynamics. Moreover, as the next Proposition shows, there are also internal symmetries
within $S_{10}$.

\begin{proposition}
\label{Prop: S1Rot}(Rotation by $\pi$ in $S_{10}$) The following identity holds:

\begin{equation}
S_{10}=T_{10}\Sigma S_{10}. \label{eq: sym4}%
\end{equation}
\label{prop_internal_symmetry}
\end{proposition}

\begin{proof}
\begin{align*}
F\left(  T_{10}\Sigma\left(  \left[ \begin{array}{c} x \\ y \end{array} \right]  \right)  \right)   &  =F\left(  \left[ \begin{array}{c} 2\pi-x \\ -y \end{array} \right]  \right) \\
&  =\left[ \begin{array}{c} 2\pi-x-2a\sin x-a\sin y \\ -y-a\sin x-2a\sin y \end{array} \right] \\
&  =\left[ \begin{array}{c} 2\pi \\ 0 \end{array} \right]  - \left[ \begin{array}{c} f\left(  x,y\right) \\ f\left(  y,x\right) \end{array} \right] \\
&  =T_{10}\Sigma\left(  F\left(  \left[ \begin{array}{c} x \\ y \end{array} \right]  \right)  \right) .
\end{align*}
\end{proof}

Proposition \ref{prop_internal_symmetry} thus indicates an internal symmetry within $S_{10}$. In fact, the composition of $T_{10}$ after  the rotation of $\pi$ is a rotation of $\pi$ around the centroid of the square $S_{10}$, implying a rotational symmetry around $\left(  \pi,0\right)  $.

We have just constructed the symmetries induced in the dynamics of $F$ by the equivariance. These show that  we need only study the map $F$ in a
bounded region, the invariant square $S_{10}$, to obtain the dynamical behavior in the full phase space $\mathbb{R}^2$. 
Equipped with all the invariance and symmetry properties established in this Section,
we now focus on the dynamics inside the invariant square $S_{10}$.


\section{The study of the dynamics restricted to $S_{10}$}

We now concentrate on the analysis of the dynamics on the invariant square $S_{10}$. We aim to prove the existence of heteroclinic curves connecting the repellers $\left(  0,0\right)
$ and $\left(  2\pi,0\right)  $, the leftmost and the rightmost vertices of
$S_{10}$, with the saddle $\left(  \pi,0\right)  $ in the centroid of $S_{10}$.
The heteroclinics, whose existence is established in Theorem \ref{Thm: Main},
 will divide $S_{10}$ in two disjoint and invariant sets
$S_{10}^{+}$ and $S_{10}^{-}$, as depicted in Figure
\ref{fig: 05}, where we obtain the heteroclinics by a numerical procedure.
Once the existence of such heteroclinics is established the dynamics of $F$ in
the invariant rectangle $S_{10}$ is fully determined, and the symmetry properties 
established in Section \ref{sec_invariance_symmetry} extend this to the full phase space $\mathbb{R}^2$.

\begin{figure}[tbh]
\centering
\includegraphics[height=4.4381in,width=4.5484in]{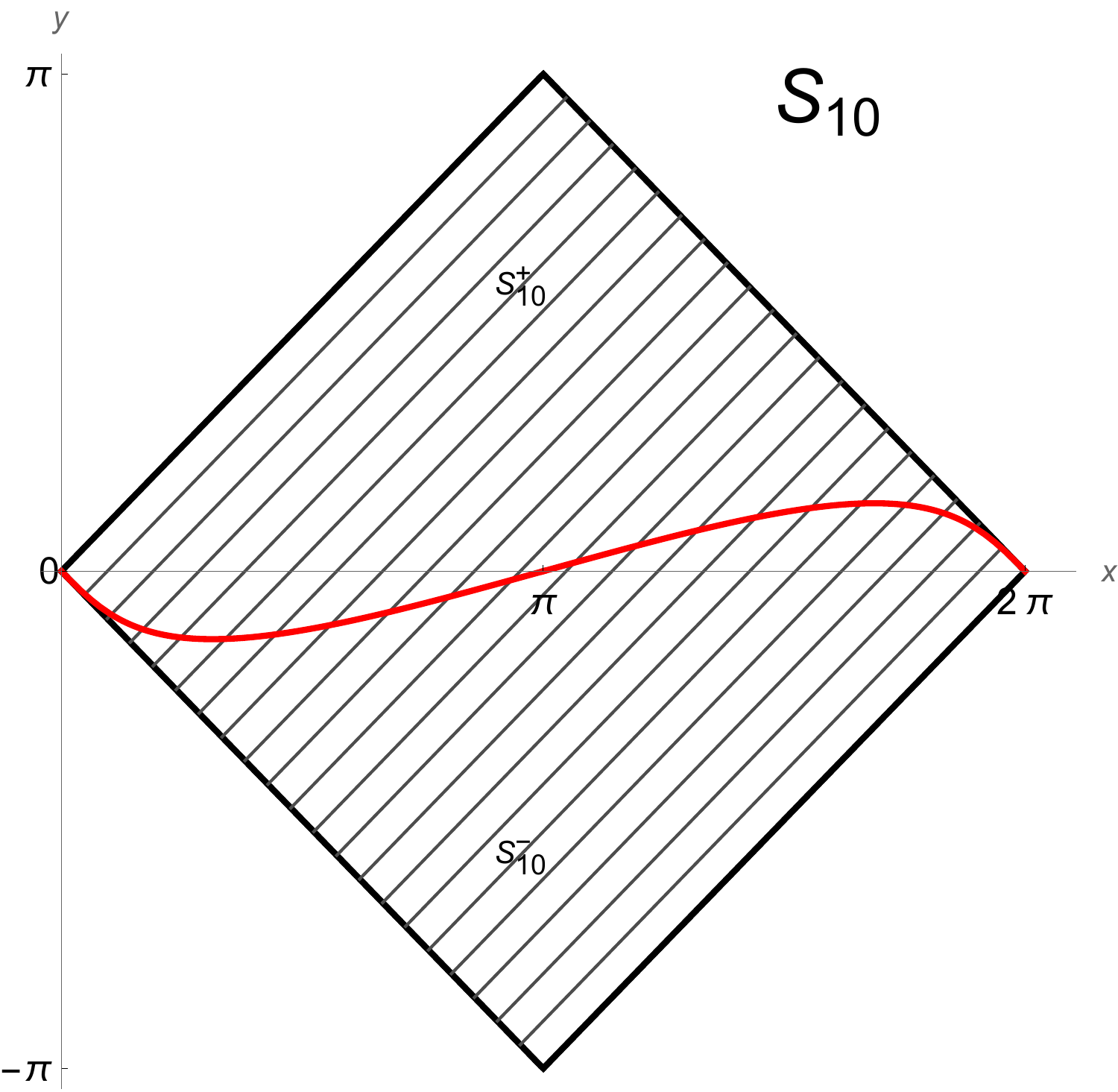}\caption{Heteroclinic 
curves connecting $\left(  0,0\right)  $ with $\left(  \pi,0\right)  $ and
$\left(  2\pi,0\right)  $ with $\left(  \pi,0\right)  $.
These heteroclinics are precisely the stable manifold of the saddle $\left(  \pi,0\right)  $. The heteroclinics above were obtained numerically with parameter $a=0.1$.}%
\label{fig: 05}%
\end{figure}

\begin{theorem}
\label{Thm: Main}There exists a heteroclinic curve connecting the unstable node in
$\left(  0,0\right)  $ with the saddle point at $\left(  \pi,0\right)  $.
\end{theorem}

\begin{theorem}
\label{Thm: Main2}There exists a heteroclinic curve connecting the saddle in
$\left(  \pi,0\right)  $ to the sink in $\left(  \pi,\pi\right)  $.
\end{theorem}

Theorem \ref{Thm: Main} is the main result in this Section, 
since it implies
that $S_{10}$ decomposes into two separately invariant regions  $S_{10}^{+}$ and $S_{10}^{-}$ bounded by the heteroclinic connections (see Figure \ref{fig: 05}).
Theorem \ref{Thm: Main2} assures that $\left(  \pi,\pi\right)$ is an attractor for
the region $s_2$ (see Figure  \ref{fig: 06}), completing the full characterization 
of the dynamics on $S_{10}$ and thus, by symmetry, to the global phase space.

We now proceed to prove Theorem \ref{Thm: Main}.

\begin{proof}
Recall, from \eqref{S10}, the definition of $S_{10}$: 
\[
S_{10}=\{(x,y)\in\mathbb{R}^{2}:-x\leq y\leq-x+2\pi\wedge x-2\pi\leq y\leq x\}.
\]
If there exists a a heteroclinic curve
$\gamma^{L}$ connecting the repeller $(0,0)$ to the saddle $(\pi,0)$, then it must
lie in the interior of $S_{10}$, since the boundary $\partial(S_{10})$ is composed
of line segments on the four invariant lines
 $\Gamma_{0}^{-}$, $\Gamma_{1}^{-}$, $\Gamma_{-1}^{+}$, and
$\Gamma_{0}^{+}$. On the other hand, the internal symmetry property 
established in Proposition \ref{Prop: S1Rot} immediately implies that, if such a 
heteroclinic 
 $\gamma^{L}$ exists, then there must also exist another heteroclinic curve 
$\gamma^{R}$ connecting the repeller $(2\pi,0)$ to
the saddle $(\pi,0)$. Finally, observe that existence of both heteroclinics  
 $\gamma^{L}$ and  $\gamma^{R}$ implies that the curve 
 $\Gamma=\gamma^{L} \cup  \gamma^{R} \cup \{(\pi,0) \}$ partitions 
$\mbox{int} (S_{10})$ into two disjoint open sets: the
upper domain $S_{10}^{+}$ above $\Gamma$ and the saddle point
 and the lower region $S_{10}^{-}$ below $\Gamma$, as depicted in Figure \ref{fig: 05}.
Observe that these open sets are separately invariant, that is 
\[F(S_{10}^{\pm
})=S_{10}^{\pm}.\]

To prove the existence of the heteroclinics and study the dynamics of $F$ 
 in the regions $S_{10}^{\pm}$, it will be convenient to partition $S_{10}$ into four smaller squares $s_{1}$, $s_{2}$, $s_{3}$, and $s_{4}$ defined by
\begin{align*}
s_{1} &  =\{(x,y)\in\mathbb{R}^{2}:x-\pi\leq y\leq x\wedge-x\leq y\leq-x+\pi\},\\
s_{2} &  =\{(x,y)\in\mathbb{R}^{2}:x-\pi\leq y\leq x\wedge-x+\pi\leq y\leq
-x+2\pi\},\\
s_{3} &  =\{(x,y)\in\mathbb{R}^{2}:x-2\pi\leq y\leq x-\pi\wedge-x\leq y\leq
-x+\pi\},\\
s_{4} &  =\{(x,y)\in\mathbb{R}^{2}:x-2\pi\leq y\leq x-\pi\wedge-x+\pi\leq
y\leq-x+2\pi\}.
\end{align*}%

\noindent These squares are formed by connecting the midpoints of the edges of the
square $S_{10}$ with its centroid at $(\pi,0)$.

\begin{figure}
[ptb]
\begin{center}
\includegraphics[
height=4.4503in,
width=4.554in
]%
{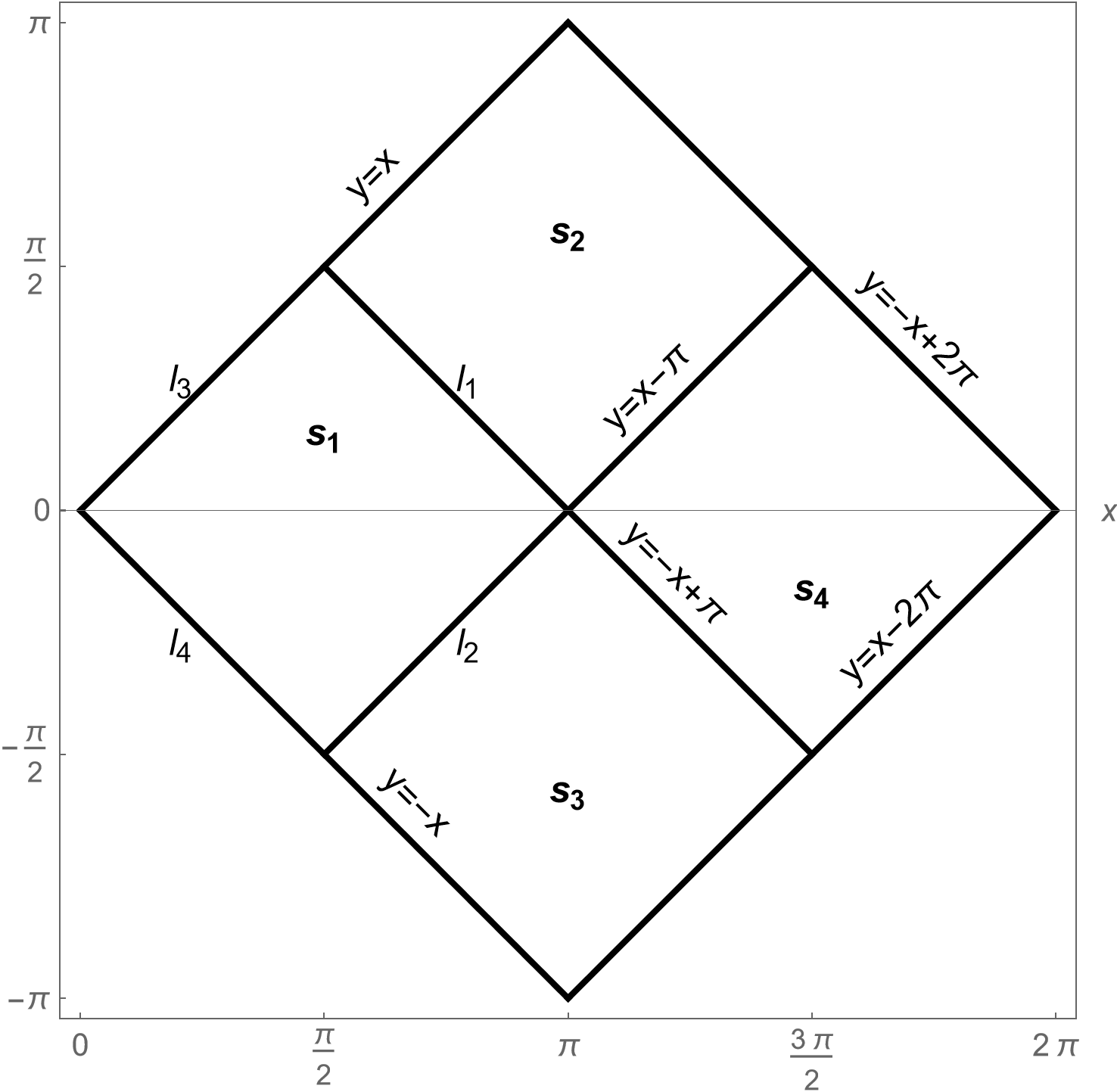}%
\caption{The four squares that cover $S_{10}$}%
\label{fig: 06}%
\end{center}
\end{figure}

Denote the top right edge of $s_{1}$ by $l_{1}$, the bottom right edge of $s_{1}$ by $l_{2}$,
the top left edge of $s_{1}$ by $l_{3}$ and the bottom right edge of $s_{1}$ by $l_{4}$,
 as depicted in Figure \ref{fig: 06}. While $l_3$ and $l_4$ lie on the invariant lines 
described by Proposition \ref{prop_translations}, it is easily seen that $l_1$ and $l_2$ are {\em not} invariant by $F$. On the edge  $l_{1}$, where $x\in\left(  \frac{\pi}{2},\pi\right)  $ and $y=\pi-x$, the mapping $F$ is given by
\[
F\left(
\begin{bmatrix}
x\\
\pi-x
\end{bmatrix}
\right)  =%
\begin{bmatrix}
x+3a\sin x\\
\pi-x+3a\sin x
\end{bmatrix}
.
\]
Since $x+3a\, \sin x>x$ and $\pi-x+3a\, \sin x>\pi-x$ for $x\in\left(  \frac{\pi
}{2},\pi\right)  $, it follows that the image $F(l_1)$ lies in the small square $s_2$, 
indicating that initial conditions on $l_1$ iterate into $s_{2}$, pointing right and 
upwards. 

A similar analysis for the lower right edge $l_{2}$, where
$x\in\left(  \frac{\pi}{2},\pi\right)  $ and $y=x-\pi$, shows that 
\[
F\left(
\begin{bmatrix}
x\\
x-\pi
\end{bmatrix}
\right)  =%
\begin{bmatrix}
x+a\sin x\\
x-\pi-a\sin x
\end{bmatrix}
,
\]
and since $x+3a\sin(x)>x$ and $x-\pi>x-\pi-a\sin(x)$ for $x\in\left(
\frac{\pi}{2},\pi\right)  $ it follows that $F(l_2)$ iterates into $s_3$. This 
means that initial conditions on $l_{2}$ also leave $s_{1}$, in this case pointing 
right and downwards.  

We now proceed to prove the existence of the heteroclinic 
$\gamma^{L}$ connecting the repeller $(0,0)$ to the saddle $(\pi,0)$. 
We focus our attention on the square $s_{1}$ since it will turn out that the
conjectured heteroclinic $\gamma^{L}$ will be contained in  $s_1$.

We first observe, from \eqref{eq_f} and \eqref{eq_varphi}, that the first component $f(x,y)$ of the diffeomorphism
$F$ satisfies 
\[
f(x,y) - x =2a\sin x+a\sin y=a \varphi(x,y).
\]
It is trivial to see that $\varphi(x,y)$ is  positive throughout $\mbox{int} (s_1)$ and the only zeros of $\varphi$ in $s_1$ are at the fixed points $(0,0)$ and $(\pi,0)$.
 Consequently, $F$ has no $\omega$-limit sets in 
$\mbox{int} (s_1)$. Since the left edges $l_3$ and $l_4$ of $s_1$, as observed above, lie on invariant lines, it follows that orbits with initial condition inside $s_1$
can only escape $s_{1}$ by crossing the interior of the right edges $l_{1}$ and $l_{2}$.



The saddle point $(\pi,0)$ is of course invariant under both $F$ and $F^{-1}$.
In a neighborhood of  $(\pi,0)$ there exists a local stable manifold $W^s_{\mbox{\scriptsize{loc}}}(\pi,0)$ 
for $F$ 
tangent to the stable 
eigenspace of $(\pi,0)$ for the
dynamics of $F$. A similar situation happens for the dynamics of $F^{-1}$ with reversed stability. 

Consider now the global stable manifold
\begin{equation}
 W^s(\pi,0) = \bigcup_{n=0}^{+\infty} F^{-n} W^s_{\mbox{\scriptsize{loc}}}(\pi,0). 
\label{eq_Ws}
\end{equation}
The intersection of this manifold with $\mbox{int}(s_1)$ is (relatively) open.

We denote by $P_* $ the boundary point $P_* = (x_*, y_*)\neq(\pi,0)$ of $W^s(\pi,0)$ in $s_1$.

Since $\varphi(x,y)$ in $\mbox{int}(s_1)$ is strictly positive, the
$xx$-projection, i.e., the projection on the direction $(1,0)$, of $W^s(\pi,0)$ lying in $\mbox{int}(s_1)$ extends  monotonically to 
the left; therefore the boundary point $P_* = (x_*, y_*)$  must lie on the boundary of $s_1$. 

The boundary of $s_1$  contains the line segments $l_3$ and $l_4$ intersecting at the point $(0,0)$. Given any point $(x_0, y_0)$ in $W^s_{loc}(\pi,0) \cap s_1$, the point $P_*$ satisfies
\[ P_* = \lim_{n \to +\infty} F^{-n} (x_0, y_0). \]
It follows that $F^{-1}(P_*) = P_*$ or equivalently, since $F$ is a diffeomorphism, 
$F(P_*) = P_*$. Thus $P_*$ must be a fixed point of $F$. However, as shown above, the segments $l_3$ and $l_4$ are invariant and do not contain fixed points. Therefore 
$P_* = (0,0)$. Thus the stable manifold $W^s(\pi,0)$ establishes the desired heteroclinic connection $\gamma^{L}$ between $(0,0)$ and $(\pi,0)$.

To conclude the construction of the heteroclinic curves in $S_{10}$, we note that $\gamma^{L}$  separates the interior of $s_{1}$ in two open sets
that we call $s_{1}^{+}$, above the heteroclinic, and $s_{1}^{-}$, below the
same heteroclinic. Again using the fact that $\varphi$ is positive in 
$\mbox{int}(s_1)$. It follows that  all initial conditions in the open
region $s_{1}^{+}$  exit that region under the iteration by crossing the
edge $l_{1}$, and  initial conditions in $s_{1}^{-}$ exit the region crossing the edge $l_{2}$. As previously stated, the internal symmetry established in Proposition \ref{Prop: S1Rot} then implies the corresponding result for the heteroclinic connection 
$\gamma^{R}$ between $(2\pi,0)$ and $(\pi,0)$ lying in $s_{4}$. Finally, the
closure of the set $\gamma^{L}\cup\gamma^{R}=\eta$ splits $S_{10}$ in two
invariant sets for the flow of $F$, that we called above $S_{10}^{+}$ and
$S_{10}^{-}$.

\end{proof}

We now turn to the proof of Theorem \ref{Thm: Main2}. Since most of the computations 
are analogous to the ones in the proof of Theorem \ref{Thm: Main},  we sketch the arguments without entering full details.

\begin{proof}
 We now consider the square $s_{2}$ depicted in Figure \ref{fig: 06}. 
This square is positively invariant for the map $F$ since the perturbation field
$\varphi$ points towards the inside of $s_{2}$ in its bottom sides, while the
top sides are $F$-invariant by virtue of Lemma \ref{lemma_1}. On the other hand, 
the second component of the perturbation field $f(y,x) - y=\varphi(y,x)$ is always positive, 
pointing to the top
vertex of $s_{2}$. 
The global setting in $s_{2}$ is thus similar to the
setting of Theorem \ref{Thm: Main}, and the orbits must approach the
top edges of $s_{2}$. The only possible accumulation point must again
be a fixed point, which is now $\left(  \pi,\pi\right)  $. Thus  we obtain a heteroclinic
connection between $\left(  \pi,0\right)  $ and $\left(  \pi,\pi\right)  $. 
\end{proof}

\begin{remark}
Note that, in the proof above,  all the orbits inside $s_{2}$ converge towards 
$\left(  \pi,\pi\right)
$. Moreover, all the initial conditions inside $s_{2}$ converge to $\left(
\pi,\pi\right)  $ under the iteration of $F$.
\end{remark}

\begin{corollary}
In the conditions of Theorems \ref{Thm: Main} and \ref{Thm: Main2}, the set $\text{int}\left(  s_{1}\right)  $ is negatively invariant , i.e., all  initial conditions inside $s_{1}$ and above the heteroclinic $\gamma^{L}$  leave this region under 
iteration of $F$, crossing towards $s_{2}$; all initial conditions below the heteroclinic
$\gamma^{L}$ inside $s_{1}$ leave this region towards $s_{3}$ under the
iteration.
\item 
\end{corollary}

\begin{proof}
The statement is an immediate consequence of the proof of Theorem \ref{Thm: Main}. 
\end{proof}

\begin{corollary} In the conditions of Theorems \ref{Thm: Main} and \ref{Thm: Main2}, the point $\left(  \pi,\pi\right)  $ is an attractor and $\text{int}\left(  s_{2}\right)  $
belongs to its basin of attraction. A similar results holds for $\left(
\pi,-\pi\right)  $ and $\text{int}\left(  s_{3}\right)  $.
\end{corollary}

\begin{proof}
This is a simple consequence of the proof of theorem \ref{Thm: Main2}. The top
vertex of $s_{2}$ is $\left(  \pi,\pi\right)  $, and all the initial
conditions inside $s_{2}$ will converge to this vertex under iteration of
$F$.
\end{proof}

\noindent The phase portrait for $a=0.1$  and the dynamics in $S_{10}^{+}$ and $S_{10}^{-}$ are depicted in
Figure \ref{Fig: 07}.

\begin{figure}
[ptb]
\begin{center}
\includegraphics[
height=4.4495in,
width=4.3911in
]%
{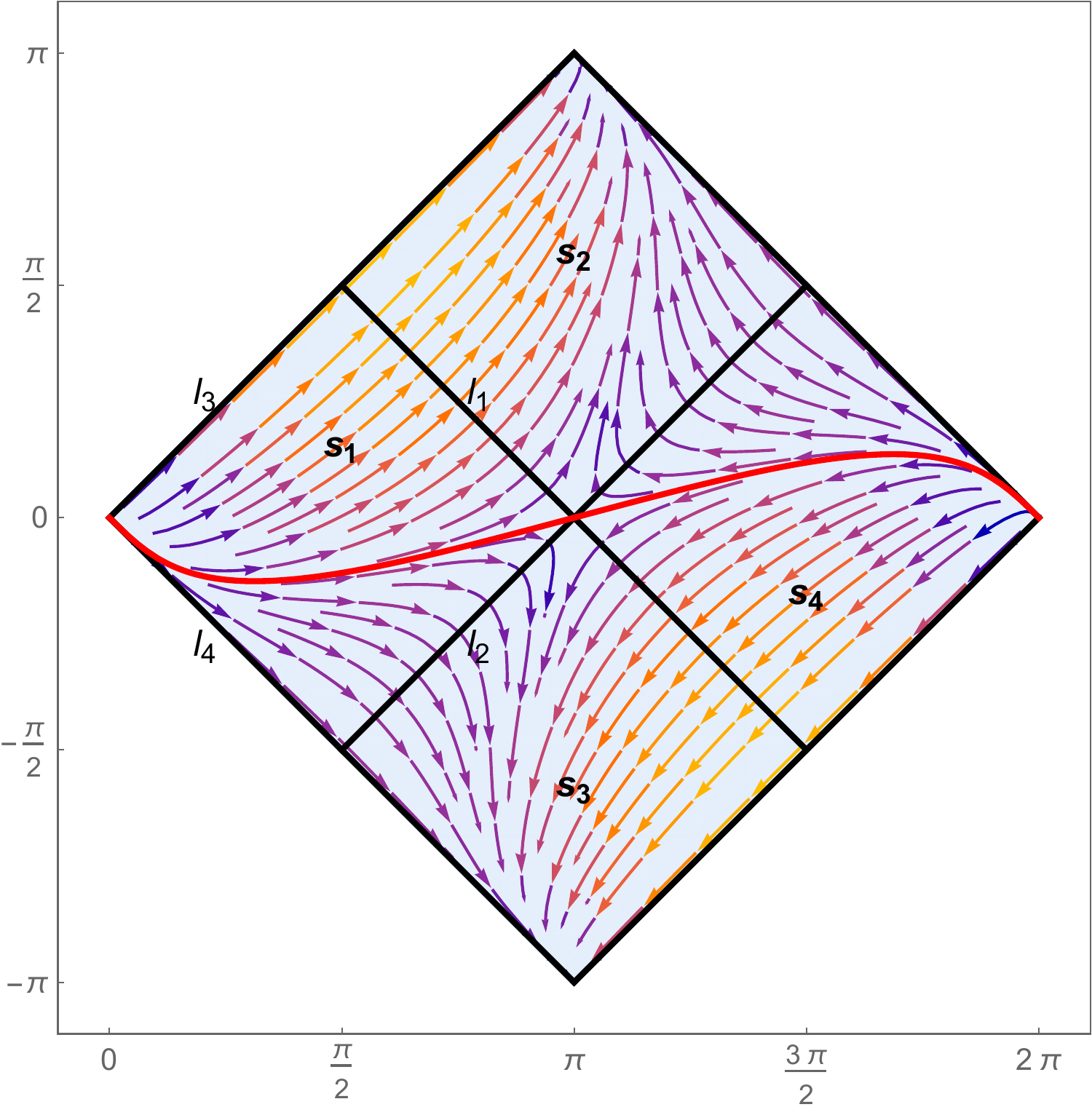}%
\caption{Dynamics in $S_{01}$ for $a=0.1$. This set is partitioned into two
separately invariant set by the curve $\eta$, closure of the union of the heteroclinics
$\gamma^{L}$ and $\gamma^{R}$.}%
\label{Fig: 07}%
\end{center}
\end{figure}

\section{Conclusion}

Theorem \ref{Thm: Main} ensures the existence of a separatrix curve $\eta=\overline{\gamma^{L}\cup\gamma^{R} \cup \{ (\pi, 0) \} }$ splitting $S_{10}$ into two separately invariant sets
$S_{10}^{+}$ and $S_{10}^{-}$. The set $\text{int}\left(  S_{10}^{+}\right)  $ above
the separatrix $\eta$ 
belongs to the basin of attraction of $\left(  \pi
,\pi\right)  $. 
Denoting by $L=\left\{  \left(
x,y\right)  :y=x,x\in\left]  0,\pi\right]  \right\} $ the  the top left edge of $S_{01}$,
the proof of Theorem \ref{Thm: Main} shows that $L$ is also in this basin of attraction. 
Applying now the symmetries (\ref{eq: sym1}), (\ref{eq: sym2}), (\ref{eq: sym3}) to
$\Xi=S_{10}\cup L$, we obtain all the rotations by $\frac{\pi}{2}k$, $k=0,1,2,3$
of $\Xi$ around $\left(  \pi,\pi\right) $ where the dynamics is equivalent to
the dynamics in $\Xi$. The union of these four sets is the basin of attraction of $\left(  \pi,\pi\right)$, giving the main result of this paper.

\bigskip

We complete the jigsaw puzzle by applying the symmetries obtained in Section
\ref{Subsubsec:symmetry}, 

which results in a complete invariant region
encircling $(\pi,\pi)$ and bounded by all the heteroclinics connecting the
points $(0,0)$, $(\pi,0)$, $(2\pi,0)$, $(2\pi,\pi)$, $(2\pi,2\pi)$, $(\pi
,2\pi)$, $(0,2\pi)$, $(0,\pi)$, and finally $(0,0)$ again. The interior of
this region has as its $\omega$-limit set the point $(\pi,\pi)$. The entire
real plane is covered by $2\pi$-periodic translations of this region, providing a tiling of the plane. The global dynamical significance of this tiling is that, modulo $2\pi$, all initial configurations not on the boundary of heteroclinics synchronize at the final state of 
mutual phase opposition, which may be considered a
rather natural result.

The complete phase diagram,  illustrated in Figure \ref{fig:06}, depicts the
heteroclinics connecting saddles to stable and unstable foci. The shadowed
bottom of the image corresponds to the region of interest for our analysis,
with the rest of the proof following from symmetry arguments.

\begin{figure}[ptb]
\begin{center}
\includegraphics[height=4.4381in,width=4.3903in]{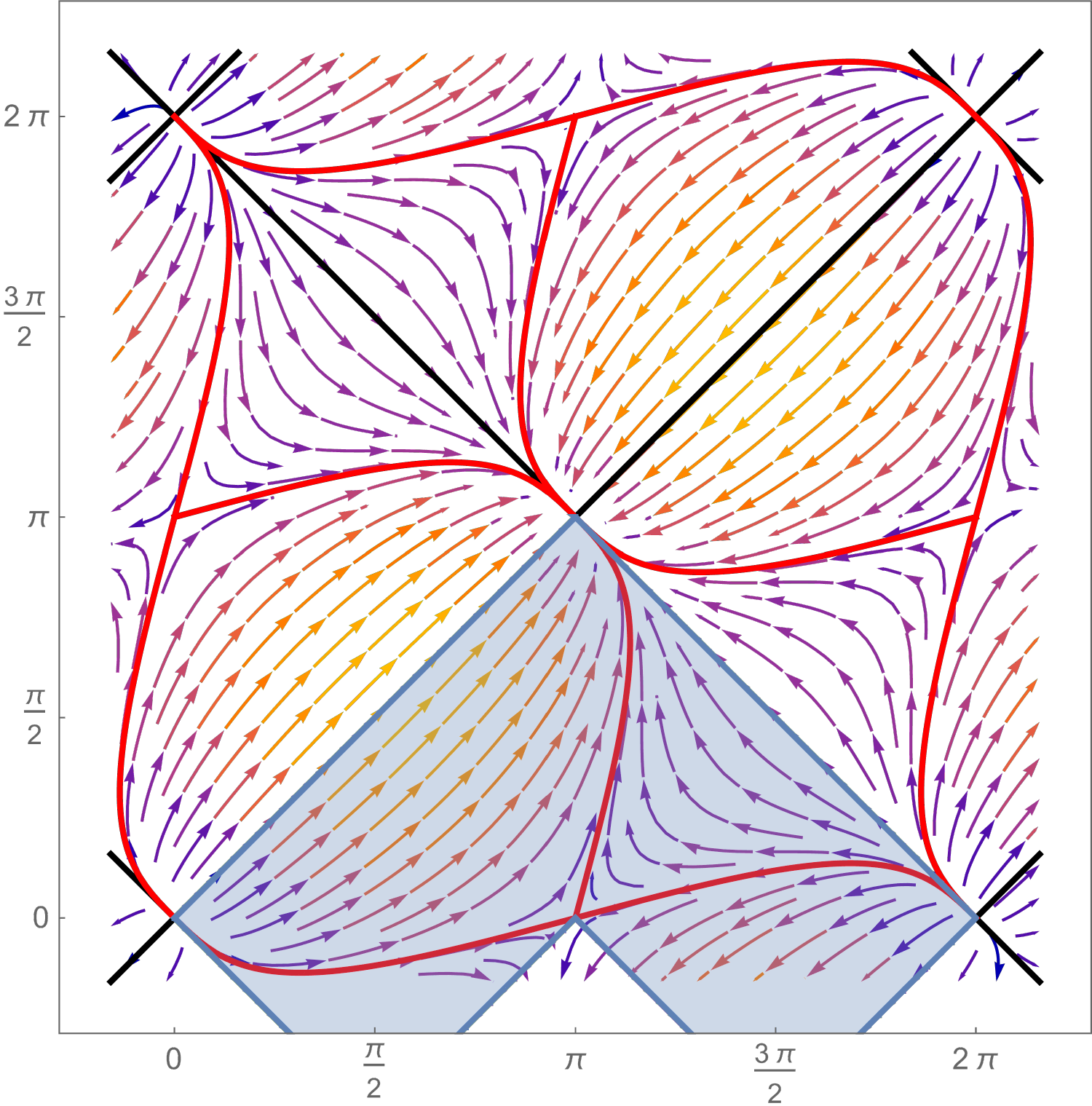}
\end{center}
\caption{Complete phase diagram with heteroclinics connecting saddles to
stable and unstable foci. Image generated numerically for $a=0.1$. We concentrate on the shaded bottom of the figure; the remainder of the proof follows by symmetry.}%
\label{fig:06}%
\end{figure}

Experimental 
verification of the proposed \textit{line model} for three
oscillators, as done in \cite{OlMe} for two oscillators, would be valuable.
Future research aims to extend these findings to oscillator strings and their
application to bi-dimensional and tri-dimensional swarms
\cite{o2017oscillators}. The outcomes hold potential for studying weakly
interacting insects or slow synchronization in neural networks using the
nearest neighbor model, akin to the integrate-and-fire models of Kuramoto
\cite{Campbell1997,Campbell1999,Kuramoto1975,Mirollo1990,Strogatz2000}.

%

\bibliographystyle{abbrv}
\bibliography{BibloH2}

\end{document}